%% file: main.tex
\documentclass{amsart}
\usepackage[margin=2cm]{geometry}
\usepackage{amssymb}
\usepackage{enumerate}   
\usepackage{graphicx}
\usepackage{mathtools}
\usepackage{tikz-cd} 
\usepackage{color}
\usepackage[colorlinks]{hyperref}

\newtheorem{theorem}{Theorem}[section]
\newtheorem{corollary}[theorem]{Corollary}

\theoremstyle{definition}
\newtheorem{definition}[theorem]{Definition}
\newtheorem{conjecture}[theorem]{Conjecture}

\newtheorem{question}[theorem]{Question}
\newtheorem{proposition}[theorem]{Proposition}
\newtheorem{example}[theorem]{Example}
\newtheorem{remark}[theorem]{Remark}
\newtheorem{lemma}[theorem]{Lemma}

\def\Zbb{\mathbb{Z}}
\def\Rbb{\mathbb{R}}
\def\Cbb{\mathbb{C}}
\def\Fbb{\mathbb{F}}

\newcommand{\G}{
	\mathrm{SL}_{2}(\mathbb{C})
}
\newcommand{\g}{
	\mathfrak{g}
}

\newcommand{\tr}{
	\mathrm{tr}\mkern 1mu
}

\newcommand{\Ad}{
	\mathrm{Ad}\rho
}

\newcommand{\rvline}{\hspace*{-\arraycolsep}\vline\hspace*{-\arraycolsep}}

\begin{document}
	
	\title[ ]{The adjoint Reidemeister torsion\\for the connected sum of knots}
	

	\author[ ]{Joan Porti}
	\address{Departament de Matem\`atiques, Universitat Aut\`onoma de Barcelona, 08193 Cerdanyola del Vall\`es, Spain}
	\email{porti@mat.uab.cat}
	
	\author[ ]{Seokbeom Yoon}
	\address{Departament de Matem\`atiques, Universitat Aut\`onoma de Barcelona, 08193 Cerdanyola del Vall\`es, Spain}
	\email{sbyoon15@mat.uab.cat}

	%    General info
	% \subjclass[2000]{Primary 54C40, 14E20; Secondary 46E25, 20C20}
	
	% \date{January 1, 2001 and, in revised form, June 22, 2001.}
	
	% \dedicatory{This paper is dedicated to our advisors.}
	
	% \keywords{Differential geometry, algebraic geometry}
	
	\begin{abstract} 
		Let $K$ be the connected sum of knots $K_1,\ldots,K_n$. It is known that the $\mathrm{SL}_2(\mathbb{C})$-character variety of the knot exterior of $K$ has a component of dimension $\geq 2$ as the connected sum admits a so-called  bending.
		We show that there is a natural way to define the adjoint Reidemeister torsion for such a high-dimensional component and prove that  it is locally constant on a subset of the character variety where the trace of a meridian is constant. We also prove that the adjoint Reidemeister torsion of $K$  satisfies the vanishing identity if  each $K_i$ does so.
	\end{abstract} 
	
	\maketitle
%	\tableofcontents
	
	\section{Introduction}
	
	Let $M$ be a compact oriented 3-manifold with tours boundary and $\mathcal{X}(M)$ be the character variety of irreducible representations  $\pi_1(M) \rightarrow \G$.
	It happens very often that $\mathcal{X}(M)$ has a component of dimension 1. 
	For instance, if the interior of $M$ admits a hyperbolic structure of finite volume, then the distinguished component is 1-dimensional \cite{thurston1979geometry} and
	 if $M$ contains no closed essential surface, then every component is 1-dimensional  \cite{cooper1994plane}.
	  
	 Once we fix a simple closed curve $\mu$ on the boundary torus $\partial M$, the \emph{adjoint Reidemeister torsion} is defined as a meromorphic function  on each 1-dimensional component of  $\mathcal{X}(M)$   under a mild assumption \cite{porti1997torsion, dubois2003torsion}.
	 It enjoys fruitful interaction with quantum field theory and carries several conjectures consequently. See, for instance, \cite{dimofte2013quantum, ohtsuki2015kashaev,  gang2019precision}.
	 Recently, it is conjectured in \cite{gang2019adjoint} that the adjoint Reidemeister torsion satisfies a certain vanishing identity with respect to  the trace function as follows.
	
	\begin{conjecture} \label{conj:main}
		Suppose that the character variety $\mathcal{X}(M)$ consists of 1-dimensional components and
		the interior of $M$ admits a hyperbolic structure of finite volume.
	Then for generic $c \in \Cbb$ we have
	\begin{equation} \label{eqn:conj}
		\sum_{[\rho] \in \mathcal{X}_\mu^c(M)} \frac{1}{\tau_\mu(M;\rho)} = 0
	\end{equation}
	where $\mathcal{X}_\mu^c(M)$ is the pre-image of $c \in \Cbb$ under the trace function $\mathcal{X}(M) \rightarrow \Cbb$ of $\mu\subset \partial M$ and  $\tau_\mu(M;\rho)$ is the adjoint Reidemeister torsion associated to $\mu$ and a representation $\rho:\pi_1(M)\rightarrow \G$.
	\end{conjecture}
	 
	As mentioned earlier, there are several  3-manifolds satisfying the conditions required in Conjecture \ref{conj:main}. 
	However, there are also several examples of 3-manifolds with torus boundary whose character varieties have high-dimensional components.  The simplest one might be (the knot exterior of) the connected sum of knots.  
	We refer to \cite{cooper1996remarks, paoluzzi2013non, chen2021rm} for other examples.
	Two immediate problems when we consider Conjecture \ref{conj:main}  for such 3-manifolds are that
	\begin{itemize}
		\item[(P1)] the adjoint Reidemeister torsion is not defined for a component of dimension $\geq 2$;
		\item[(P2)]  the sum in the equation \eqref{eqn:conj} does not make sense as the level set $\mathcal{X}_\mu^c(M)$ is no longer finite.
	\end{itemize}
	Related to these problems, we address the following question.
	\begin{question}\label{ques}
	    Is the adjoint Reidemeister torsion defined and locally constant on $\mathcal{X}_\mu^c(M)$?
	\end{question}
	\noindent If the answer of Question \ref{ques} is positive, then the sum in the equation \eqref{eqn:conj} makes sense for $M$ in an obvious way: by taking one representative on each connected component of $\mathcal{X}_\mu^c(M)$.

	The main purpose of the paper is to investigate Question \ref{ques} and Conjecture \ref{conj:main} for the connected sum of knots. Let $K$ be the connected sum of knots $K_1,\ldots,K_n$ in $S^3$ and $\mu$ be a meridian.
	We denote by $M$ and $M_j$ the knot exteriors of $K$ and $K_j$, respectively.
	For technical reasons, we assume that for $1 \leq j \leq n$
%	\begin{flushleft}
%		(A) the level set $\mathcal{X}_\mu^c(M_j)$ consists of finitely many $\mu$-regular characters with the canonical longitude having trace other than $\pm2$  for generic $c \in \Cbb$.
%	\end{flushleft}
	\begin{itemize} 
%		\item[(a)] the character variety $\mathcal{X}(M_j)$ consists of 1-dimensional components;
		\item[(C)] the level set $\mathcal{X}_\mu^c(M_j)$ consists of finitely many $\mu$-regular characters with the canonical longitude having trace other than $\pm2$  for generic $c \in \Cbb$. 
	\end{itemize}
	For example, one may choose $K_j$ as a two-bridge knot or a torus knot.
	It is known that the character variety $\mathcal{X}(M)$  has a component of dimension $\geq 2$ as the connected sum admits a so-called \emph{bending}.	
	We refer to \cite{johnson1987deformation, paoluzzi2013non, kitano2020finiteness} for details on the bending construction.

	\begin{theorem} \label{thm:main0}
	Let $K$ be the connected sum of knots $K_1,\ldots, K_{n}$  satisfying the above condition (C) and $\mu$ be a meridian.
	Then there is a natural way to define the adjoint Reidemeister torsion on $\mathcal{X}^c_\mu(M)$ for generic $c \in \Cbb$ which is locally constant.
	\end{theorem} 
	
	\begin{theorem} \label{thm:main} Let $K$ be the connected sum of knots $K_1,\ldots, K_{n}$  satisfying the above condition (C) and $\mu$ be a meridian. Then the knot exterior $M$ of $K$ satisfies the equation \eqref{eqn:conj}  if each $M_j$ does so.
	\end{theorem}
	It is proved in \cite{yoon2020adjoint} that every hyperbolic two-bridge knot satisfies the equation \eqref{eqn:conj} for a meridian. We thus obtain the following corollary.
	\begin{corollary} \label{cor}
		 The knot exterior of the connected sum of hyperbolic two-bridge knots  satisfies the equation \eqref{eqn:conj} for a meridian.
	\end{corollary}	
	
	\begin{remark} Conjecture \ref{conj:main} was derived from the 3d-3d correspondence under 
		the assumption that the interior of $M$ admits a hyperbolic structure. We refer to  \cite[Section 3]{gang2019adjoint} for details.
		It fails without the assumption since torus knot exteriors do not satisfy the equation \eqref{eqn:conj}. However, Theorem \ref{thm:main} and Corollary \ref{cor} suggest that one can relax the hyperbolicity condition, as the connected sum of knots is never hyperbolic.	
	\end{remark}

	The paper is organized as follows. In Section \ref{sec:2},  we briefly recall  basic definitions on the sign-refined Reidemeister torsion. 
	We define the adjoint Reidemeister torsion for the connected sum of knots in Sections  \ref{sec:connector} and \ref{sec:gluing}, and prove Theorems \ref{thm:main0} and \ref{thm:main} in Section \ref{sec:3.3}.

	\section{Review on the sign-refined Reidemeister torsion} \label{sec:2}
	\subsection{The Reidemeister torsion of a chain complex}
	Let $C_\ast$ be a chain complex of vector spaces over a field $\Fbb$
	\[C_\ast=(0 \rightarrow C_n \overset{\mkern-5mu\partial_n}{\longrightarrow}  \cdots \longrightarrow C_1 \overset{\mkern-5mu\partial_1}{\longrightarrow} C_0\rightarrow 0)\]
	and $H_\ast(C_\ast)$ be the homology of $C_\ast$.
	For a basis $c_\ast$ of $C_\ast$ and a basis  $h_\ast$  of $H_\ast(C_\ast)$ the \emph{Reidemeister torsion}  is defined as follows. 
	Here and throughout the paper, every basis and tuple is assumed to be ordered.
	For $0 \leq i \leq n$ we choose a lift $\widetilde{h}_i$ of $h_i$ to $C_i$ and a tuple $b_i$ of vectors in $C_i$ such that $\partial_i b_i$ is a basis of $\partial_i C_i$.
	Then the tuple $c'_i=(\partial_{i+1} b_{i+1}, \widetilde{h_i}, b_i)$ is another basis of $C_i$.
	Letting $A_i$ be the basis transition matrix taking $c_i$ to $c'_i$,  we have
	\begin{equation*}
		\mathrm{tor}(C_\ast,c_\ast, h_\ast) = \prod_{i=0}^n  \det A_i^{(-1)^{i+1}} \in \Fbb^\ast.
	\end{equation*}
Also, the \emph{sign-refined Reidemeister torsion} is defined as
	\begin{equation*}
	\mathrm{Tor}(C_\ast,c_\ast, h_\ast) =  (-1)^{|C_\ast|}\, \mathrm{tor}(C_\ast,c_\ast,h_\ast) \in \Fbb^\ast, 	\quad	|C_\ast|=\sum_{i=0}^n \alpha_i(C_\ast) \beta_i(C_\ast)
	\end{equation*}
	where $\alpha_i(C_\ast) = \sum_{j=0}^i \dim C_j$ and $\beta_i(C_\ast) = \sum_{j=0}^i \dim H_j(C_\ast)$.

	Suppose that we have a short exact sequence of chain complexes
	\begin{equation}  \label{eqn:ses}
	0\rightarrow C'_\ast \rightarrow C_\ast \rightarrow C_\ast''\rightarrow 0
	\end{equation}
	with  bases  $c _\ast, c_\ast',$ and $c_\ast''$ of $C_\ast, C_\ast'$, and $C_\ast''$, respectively.	
%	Here $\widetilde{c}''_\ast$ is a lift of $c''_\ast$ to $C_\ast$.
	It is proved in \cite[Lemma 3.4.2]{turaev1986reidemeister} that if $c_\ast, c'_\ast,$ and $c''_\ast$ are \emph{compatible} with respect to the sequence \eqref{eqn:ses}, i.e.,  $c_\ast = (c'_\ast, c''_\ast)$, then
	\begin{equation} \label{eqn:gluing}
		\mathrm{Tor}(C_\ast,c_\ast,h_\ast) = (-1)^{v+u}\, \mathrm{Tor}(C'_\ast,c'_\ast,h'_\ast)\,\mathrm{Tor}(C''_\ast,c''_\ast,h''_\ast)\,\mathrm{tor}(\mathcal{H})
	\end{equation}
	where $h_\ast, h_\ast',$ and $h_\ast''$ are bases of $H_\ast(C_\ast), H_\ast(C_\ast')$, and $H_\ast(C_\ast'')$, respectively.
	 Here
	 \begin{align}
	 	\label{eqn:v} v & = \sum_{i} \alpha_{i-1}(C'_\ast) \alpha_i(C''_\ast), \\
	 	\label{eqn:u} u & = \sum_{i} \left( (\beta_i(C_\ast)+1)(\beta_i(C'_\ast)+\beta_i(C''_\ast)) + \beta_{i-1}(C'_\ast) \beta_i(C''_\ast) \right),
	 \end{align} and
	$\mathrm{tor}(\mathcal{H})$ is the Reidemeister torsion of the long exact sequence induced from \eqref{eqn:ses} with respect to $h_\ast, h'_\ast$, and $h_\ast''$. 
	 We refer to \cite{turaev1986reidemeister,turaev2002torsions} for details.

	\subsection{The adjoint Reidemeister torsion of a CW-complex} \label{sec:cw}
	
	Let $\g$ be the Lie algebra of $\G$ and fix a basis of $\g$ as 
	\begin{equation*}
		e_1 = \begin{pmatrix} 0 & 1 \\ 0 & 0\end{pmatrix},\ 
		e_2 = \begin{pmatrix} 1 & 0 \\ 0 & -1\end{pmatrix},\ 
		e_3 = \begin{pmatrix} 0 & 0 \\ 1 & 0\end{pmatrix}.
	\end{equation*}
	Note that the Killing form $\langle \cdot, \cdot \rangle$ on $\g$ is given by
	\begin{equation*}
		\left\langle \begin{pmatrix} b & a \\ c & -b \end{pmatrix}, \begin{pmatrix} b' & a' \\ c' & -b' \end{pmatrix}  \right\rangle = 8  b b' +4 a c'+4 ca'.
	\end{equation*}			
	Let $X$ be a finite CW-complex and $\rho:\pi_1(X)\rightarrow \G$ be a representation. 
	We consider a cochain complex 
	\[ C^\ast(X;\g_\rho) = \mathrm{Hom}_{\Zbb[\pi_1 \mkern-2mu X]} \left(C_\ast(\widetilde{X};\Zbb), \g\right)\]
	where  $\widetilde{X}$ is the universal cover of $X$.  Here $\g $ is viewed as a $\Zbb[\pi_1 (X)]$-module through the adjoint action $\mathrm{Ad}\rho : \pi_1(X)\rightarrow \mathrm{Aut}(\g)$ associated to $\rho$. 
	We denote the cohomology   of $C^\ast(X;\g_\rho)$  by  $H^\ast(X;\g_\rho)$ and call it the \emph{twisted cohomology}.  Note that  $H^0(X;\g_\rho)$ coincides with the set of invariant vectors in $\g$ under the $\pi_1(X)$-action.
	
	Let $c_1,\ldots,c_n$ be all the cells of $X$ and fix their order by $c_X=(c_1,\ldots,c_n)$. 
	We assume that each cell $c_i$ has a preferred orientation and a lift $\widetilde{c}_i$ to $\widetilde{X}$.
	We define an element ${c}_i^{(k)}\in C^\ast(X;\g_\rho)$ for $1 \leq i\leq n$ and $1 \leq k \leq 3$ by assigning  $\widetilde{c}_i$ to $e_k$ and every cell of $\widetilde{X}$ that is not a lift of $c_i$ to $0$.  Then  the tuple
	\[\mathbf{c}_X=\left(c_1^{(1)},c_1^{(2)},c_1^{(3)}, \ldots, c_n^{(1)},c_n^{(2)},c_n^{(3)}\right)\]  is a basis of $C^\ast(X;\g_\rho)$. We refer to it as the \emph{geometric basis}.

	Let $C_\ast(X;\Rbb)$ be the ordinary chain complex of $X$ with the real coefficient.
	Note that the tuple $c_X$ is a basis of $C_\ast(X;\Rbb)$.
	For  an orientation  $o_X$ of  the $\Rbb$-vector space $H_\ast(X;\Rbb)$ we define
	\[\epsilon(o_X) = \mathrm{sgn} \left(\mathrm{Tor}(C_\ast(X;\Rbb), c_X, h_X) \right) \in \{ \pm 1\}\]
	where $h_X$ is any basis of $H_\ast(X;\Rbb)$ positively oriented with respect to $o_X$
	and $\mathrm{sgn}(x)$ is the sign of $x\in\Rbb^\ast$.
	
	\begin{definition}
		For a basis $\mathbf{h}_X$ of $H^\ast(X;\g_\rho)$ and an orientation $o_X$ of $H_\ast(X;\Rbb)$ the \emph{adjoint Reidemeister torsion} is defined as
	\[\tau(X;\rho,\mathbf{h}_X,o_X)= \epsilon(o_X) \cdot \mathrm{Tor} \left( C^\ast(X;\g_\rho), \mathbf{c}_X,\mathbf{h}_X \right) \in \Cbb^\ast.\]
	The above definition does not depend on the order, orientations, and lifts of $c_i$'s.  Moreover, it does not depend on the choice of a basis of $\g$ if the Euler characteristic of $X$ is zero
	\end{definition}
	Note that every notion in this section associated to $\rho$  is invariant under conjugating $\rho$ up to an appropriate isomorphism. In particular, the adjoint Reidemeister torsion is invariant under the conjugation.
%    	 We define
%    	\begin{equation} \label{eqn:tau}
%    		\tau_X(\rho;h) = \mathrm{Tor}(C^\ast(X;\g_\rho), \mathcal{C}_X,h))
%    	\end{equation}
%    	where $h$ is a basis of $H^\ast(X;\g_\rho)$.
%    	Note that the equation \eqref{eqn:tau} does not depend on the choice of lifts of cells in $X$.

\begin{example} \label{exam:torus}
	Let $\Sigma$ be a 2-torus with a usual CW-structure: one 0-cell $p$, two 1-cells $\mu$ and $\lambda$, and one 2-cell $\Sigma$ as in Figure \ref{fig:torus}~(left). We choose their lifts (to the universal cover of $\Sigma$) as in Figure \ref{fig:torus}~(right)
	and fix an order of the cells by $c_\Sigma=(p,\mu,\lambda,\Sigma)$. 
	Let $o_\Sigma$ be the orientation of $H_\ast(\Sigma;\Rbb)$ induced  from $c_\Sigma$ so that $\epsilon(o_\Sigma) = 1$.
	\begin{figure}[!h]
		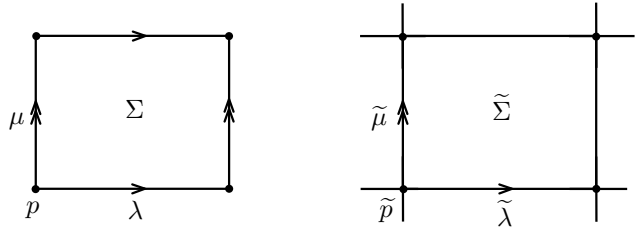
		\caption{The cells of a 2-torus and their lifts.}
		\label{fig:torus}
	\end{figure}

Let $\rho:\pi_1(\Sigma)\rightarrow \G$ be a representation with $\tr \rho(\mu)\neq \pm2$. Up to conjugation we have
	\[\rho(\mu)=\begin{pmatrix}	
		m & 0 \\
		0 & m^{-1}
	\end{pmatrix}, \quad \rho(\lambda)=\begin{pmatrix}
		l & 0 \\
		0 & l^{-1}
	\end{pmatrix}\]
	for some $m \neq \pm1$ and $l \in \Cbb^\ast$. 
	% 	The geometric basis of $C^i(\Sigma;\g_\rho)$ is given by $\{p^{(1)},p^{(2)},p^{(3)}\}$ for $i=0$, $\{\mu^{(1)},\mu^{(2)},\mu^{(3)}, \lambda^{(1)}, \lambda^{(2)}, \lambda^{(3)}\}$ for i=1, and $\{\Sigma^{(1)},\Sigma^{(2)},\Sigma^{(3)}\}$ for $i=2$. 
	With respect to the geometric basis, the boundary maps  $\delta^0 : C^0(\Sigma;\g_\rho) \rightarrow C^1(\Sigma;\g_\rho)$ and $\delta^1 : C^1(\Sigma;\g_\rho) \rightarrow C^2(\Sigma;\g_\rho)$  are given by
	\[ \delta^0 = \begin{pmatrix}  
		\mathrm{Ad}\rho(\mu) - I_3 \\
		 \mathrm{Ad}\rho(\lambda)-I_3
	\end{pmatrix}, \quad 
	\delta^1 = \begin{pmatrix}
		\mathrm{Ad}\rho(\lambda) - I_3 &  I_3 - \mathrm{Ad}\rho(\mu)
		\end{pmatrix}.
	\]
	Here $I_k$ is the  identity matrix of size $k$.
	It follows that $ \dim H^i(\Sigma;\g_\rho)=1$ for $i=0,2$, $\dim H^i(\Sigma;\g_\rho)=2$ for $i=1$, and $\dim H^i(\Sigma;\g_\rho)=0$ otherwise.
	Let $P=\frac{1}{8} e_2 \in H^0(\Sigma;\g_\rho)$ and define maps
		\begin{align*}
	& \psi^0 : C^0(\Sigma;\g_\rho) \rightarrow \Cbb,\ \alpha \mapsto \langle \alpha(\widetilde{p}), P \rangle,\\
	& \psi^1 : C^1(\Sigma;\g_\rho) \rightarrow \Cbb^2, \ \alpha \mapsto \left(\langle \alpha(\widetilde{\mu}), P \rangle,\, \langle \alpha(\widetilde{\lambda}), P \rangle \right),\\
	& \psi^2 : C^2(\Sigma;\g_\rho) \rightarrow \Cbb,  \ \alpha \mapsto \langle \alpha(\widetilde{\Sigma}), P \rangle.
	\end{align*}  
	One easily checks that each map $\psi^i$ induces
	an isomorphism $H^i(\Sigma;\g_\rho) \rightarrow \Cbb$ ($\Cbb^2$ if $i=1$). For simplicity we use the same notation $\psi^i$ for these isomorphisms.
	We choose a basis   $\mathbf{h}_\Sigma^i$  of $H^i(\Sigma;\g_\rho)$ by the pre-image of the standard basis of $\Cbb$   ($\Cbb^2$ if $i=2$) under  $\psi^i$. 
	Explicitly, we have  $\mathbf{h}_\Sigma^0 = p^{(2)}$, $\mathbf{h}_\Sigma^1 = ( \mu^{(2)} , \lambda^{(2)})$, and $\mathbf{h}_\Sigma^2 = \Sigma^{(2)}$.
	 Choosing a tuple $b^i$ of vectors in $C^i(\Sigma;\g_\rho)$ as $b^0 = (p^{(1)},p^{(3)})$, $ b^1 = (\lambda^{(1)},\lambda^{(3)})$, and $b^2 = \emptyset$, we obtain
	\begin{align*}
		\tau(\Sigma;\rho,\mathbf{h}_\Sigma,o_\Sigma) &= - 1 \cdot (m^2-1)(m^{-2}-1) \cdot (-(m^2-1)(m^{-2}-1))^{-1} =1.
	\end{align*}
	Note that a different choice of $P \in H^0(\Sigma;\g_\rho)$ changes the basis $\mathbf{h}_\Sigma$ but still we have $\tau(\Sigma;\rho,\mathbf{h}_\Sigma,o_\Sigma)=1$.
	\end{example}

	\subsection{The adjoint Reidemeister torsion of a knot exterior} \label{sec:knot}

	Let $M$ be the knot exterior of a knot $K \subset S^3$ with any given triangulation.
	It is well-known that $\dim H_i(M;\Rbb)=1$ for $i=0,1$ and  $\dim H_i(M;\Rbb)=0$ otherwise. 
	We choose  the orientation $o_M$ of $H_\ast(M;\Rbb)$ induced from a basis $h_M=(pt, \mu)$ of $H_\ast(M;\Rbb)$ where $pt$ is a point in $M$ and $\mu$ is a meridian of $K$ oriented arbitrarily. 
	
	Let $\rho:\pi_1(M)\rightarrow \G$ be a representation of the knot group.
	For the sake of simplicity, we assume that 
	\[ m \neq \pm1 \textrm{ and } \Delta_K(m^2) \neq 0\]
	where $m$ is an eigenvalue of $\rho(\mu)$ and $\Delta_K$ is the Alexander polynomial of $K$. 
	It follows that if $\rho$ is reducible, then it should be abelian (see e.g. \cite{burde2013knots}). Therefore,  $\rho$ is either irreducible (Section \ref{sec:irred}) or abelian (Section \ref{sec:red}).
	
	\subsubsection{Irreducible representations} \label{sec:irred}
	Suppose that $\rho$ is irreducible. In this case we further assume that $\rho$ is  $\mu$-regular \cite[Definition 3.21]{porti1997torsion}, i.e., $\dim H^1(M;\g_\rho) =1$ and 
 the inclusion $\mu \hookrightarrow M$ induces an injective map $H^1(M;\g_\rho) \rightarrow H^1(\mu;\g_\rho)$.
%  (In fact, we would not use this precise definition in the paper).5
%  , since most of irreducible representations are $\mu$-regular in general and we are interested in generic ones.
	We choose an element $P \in H^0(\Sigma;\g_\rho)$, where $\Sigma = \partial M$, and define maps
	\begin{align*}
	& \psi^1 : C^1(M;\g_\rho) \rightarrow \Cbb, \ \alpha \mapsto \langle \alpha(\widetilde{\mu}), P \rangle, \\
	& \psi^2 : C^2(M;\g_\rho) \rightarrow \Cbb,  \ \alpha \mapsto \langle \alpha(\widetilde{\Sigma}), P \rangle,
	\end{align*}  
	where $\widetilde{\mu}$ and $\widetilde{\Sigma}$ are lifts of $\mu$ and $\Sigma$ (to the universal cover of $M$) respectively satisfying  $\widetilde{\mu} \subset \widetilde{\Sigma}$. Here the boundary torus $\Sigma$ is oriented as in Stokes' theorem.
	It is proved in \cite{porti1997torsion} that the $\mu$-regularity implies that $\psi^i$ induces an isomorphism $H^i(M;\g_\rho) \rightarrow \Cbb$ for $i=1,2$.  We define
	\begin{equation}\label{eqn:torog}
		\tau_\mu(M;\rho) = \tau(M;\rho,\mathbf{h}_M,o_M) 
	\end{equation}
	where $\mathbf{h}_M^i$ is a basis of $H^i(M;\g_\rho)$ given by the pre-image of the standard basis of $\Cbb$ under $\psi^i$.
	Note that a different choice of $P \in H^0(\Sigma;\g_\rho)$ changes the basis $\mathbf{h}_M$ but not the value of $\tau_\mu(M;\rho)$.

	\subsubsection{Abelian representations}  \label{sec:red}
	Suppose that $\rho$ is abelian. This case might not be that interesting, as it essentially reduces to the case of Alexander polynomial. We however present explicit setups  here for Section \ref{sec:3}. 
	
	\begin{lemma} \label{lem:abel}
		We have $\dim H^i(M;\g_\rho)=1$ for $i=0,1$ and $\dim H^i(M;\g_\rho)=0$ otherwise.
	\end{lemma}
	\begin{proof}
	We choose any Wirtinger presentation of  the knot group
	\[\pi_1(M) = \langle g_1, \ldots, g_n \, | \, r_1,\ldots,r_{n-1} \rangle.\]
	 Recall that the corresponding 2-dimensional cell complex $X$ consists of 
 one 0-cell $p$, $n$ 1-cells $g_1,\cdots,g_n$, and $n-1$ 2-cells $r_1,\cdots,r_{n-1}$.
	It is known that $X$ is simple homotopic equivalent to $M$ and thus we may use $X$ instead of $M$.
	We choose a lift  of the base point $p$  arbitrarily and the lifts of other cells accordingly.
	Then with respect to the geometric basis, the boundary maps $\delta^0 :C^0(X;\g_\rho) \rightarrow C^1(X;\g_\rho)$ and $\delta^1 :C^1(X;\g_\rho) \rightarrow C^2(X;\g_\rho)$ are given as
	\[\delta^0 = \begin{pmatrix}
		\Phi(g_1-1) \\
		\vdots\\
		\Phi(g_n-1)
	\end{pmatrix}, \quad \delta^1 =
 \begin{pmatrix}
		\Phi( \frac{\partial r_1}{\partial g_1}) & \cdots & \Phi( \frac{\partial r_1}{\partial g_n}) \\
		\vdots & \ddots & \vdots	\\
		\Phi(\frac{\partial r_{n-1}}{\partial g_1}) & \cdots & \Phi(\frac{\partial r_{n-1}}{\partial g_n})
	\end{pmatrix}
	\] where $\Phi$ is the $\Zbb$-linear extension of $\mathrm{Ad}\rho$ and $\partial r_j/\partial g_i$ denotes the Fox free differential.
	Recall that  up to conjugation
		\[\rho(g_1)=\cdots=\rho(g_n) = \begin{pmatrix} m & 0 \\ 0 & m^{-1} \end{pmatrix}, \  m \neq \pm 1.\]
		It is clear that $\mathrm{Im} \, \delta^0 \simeq \Cbb^2$, $\mathrm{Ker} \, \delta^0 \simeq \Cbb$ and $\dim H^0(X;\g_\rho)=1$.
		On the other hand, $\delta^1$ is surjective  since $\Delta_K(1) \neq 0$ and $\Delta_K(m^{\pm2}) \neq 0$.
%		 (see also the proof of Lemma \ref{lem:tauabel} below). 
		 It follows that   $\dim H^2(X;\g_\rho)=0$ and   $\dim H^1(X;\g_\rho) =1$ since the Euler characteristic of $X$ is zero.
		Explicitly, the twisted cohomology of $X$ is generated by
		\begin{equation} \label{eqn:xbasis}
			C^0(X;\g_\rho) \ni \alpha \textrm{ s.t. } \alpha(\widetilde{p})=e_2, \quad C^1(X;\g_\rho) \ni \alpha  \textrm{ s.t. }  \alpha(\widetilde{g_i})=e_2 \ \forall i. 
		\end{equation}
	\end{proof}

	Once again, we choose an element $P \in H^0(\Sigma;\g_\rho)=H^0(M;\g_\rho)$ and define
		\begin{align*}
		& \psi^0 : C^0(M;\g_\rho) \rightarrow \Cbb, \ \alpha \mapsto \langle \alpha(\widetilde{p}), P \rangle, \\
		& \psi^1 : C^1(M;\g_\rho) \rightarrow \Cbb,  \ \alpha \mapsto \langle \alpha(\widetilde{\mu}), P \rangle,
	\end{align*}  
	where $\widetilde{p}$ and $\widetilde{\mu}$ are lifts of $p$ and $\mu$  (to the universal cover of $M$) respectively satisfying  $\widetilde{p} \subset \widetilde{\mu}$.
	It is clear from the equation \eqref{eqn:xbasis} that  $\psi^i$ induces an isomorphism $H^i(M;\g_\rho) \rightarrow \Cbb$ for $i=0,1$.
	We define
	\begin{equation*}
		\tau_\mu(M;\rho) = \tau(M;\rho,\mathbf{h}_M,o_M) 
	\end{equation*}
	where $\mathbf{h}_M^i$ is a basis of $H^i(M;\g_\rho)$ given by the pre-image of the standard basis of $\Cbb$ under $\psi^i$.
%	Note that a different choice of $P \in H^0(\Sigma;\g_\rho)$ changes the basis $\mathbf{h}_M$ but not the value of $\tau_\mu(M;\rho)$.
	In fact, one can compute that
	\[\tau_\mu(M;\rho)=  \frac{\Delta_K(m^2) \Delta_K(m^{-2})}{(m-m^{-1})^2}\]
	up to sign, but we would not use this fact in this paper.
%\begin{lemma} \label{lem:tauabel} We have
%	\[\tau_\mu(M;\rho)=\epsilon \cdot \frac{\Delta_K(m^2) \Delta_K(m^{-2})}{(m^2-1)(m^{-2}-1)}\]
%	for some $\epsilon \in \{\pm1\}$.
%%	Note that the Alexander polynomial $\Delta_K(t)$ is well-defined up to $\pm t^n$ but the above equation is well-defined.
%\end{lemma}
%\begin{proof} Recall the proof of Lemma \ref{lem:abel} that we may use the cell complex $X$ instead $M$.
%Let $P=\frac{1}{8} e_2 \in H^0(X;\g_\rho)$ so that the basis $\mathbf{h}_M$ coincides with the one described in the equation \eqref{eqn:xbasis}.
%	Choosing a sequence $b^i$ of vectors in $C^i(X;\g_\rho)$ as 
%\[b^0 = (p^{(1)}, p^{(3)}),\ b^1 =(g_2^{(1)},g_2^{(2)},g_2^{(3)},\cdots,g_{n}^{(1)},g_{n}^{(2)},g_{n}^{(3)}),\ b^2=\emptyset,\]  a straightforward computation gives
%\begin{align*}
%	 \tau_\mu(M;\rho) &= - \det
%\begin{pmatrix}
%	m^2-1 & 0 & \rvline & 0 & \rvline & \\
%	0 & 0 & \rvline & 1 & \rvline &  0 \\
%	0 &  m^{-2}-1 & \rvline & 0 & \rvline & \\
%	\hline
%	\vdots & \vdots & \rvline & \vdots & \rvline &\\
%	m^2-1 & 0 & \rvline & 0  & \rvline & \\
%	0 & 0 & \rvline & 1 & \rvline & I_{3n-3} \\
%	0 &  m^{-2}-1 & \rvline  & 0 & \rvline &
%\end{pmatrix}^{-1} \mkern-15mu \det 
% \begin{pmatrix}
%	\Phi( \frac{\partial r_1}{\partial g_2}) & \cdots & \Phi( \frac{\partial r_1}{\partial g_{n}}) \\
%	\vdots & \ddots & \vdots	\\
%	\Phi(\frac{\partial r_{n-1}}{\partial g_2}) & \cdots & \Phi(\frac{\partial r_{n-1}}{\partial g_{n}})
%\end{pmatrix} = \pm \frac{\Delta_K(m^2) \Delta_K(m^{-2})}{(m^2-1)(m^{-2}-1)}.
%\end{align*}
%\end{proof}
	
		\section{The connected sum of knots} \label{sec:3}
	
		Let $K$ be the connected sum of knots $K_1,\ldots,K_n$ in $S^3$. 
		We denote by $M$ and $M_j$ the knot exteriors of $K$ and $K_j$, respectively.
		It is known that the JSJ decomposition of $M$ consists of a composing space and $M_1,\ldots,M_n$.

		\subsection{A composing space} \label{sec:connector}
		Let $D_1,\ldots,D_n$ be mutually disjoint discs in the interior of a disc $D^2$ and 
		let $W= D^2 \setminus \mathrm{int}(D_1 \sqcup \cdots \sqcup D_{n})$ be a planar surface.
		Here $\mathrm{int}(X)$ denotes the interior of $X$.
		A \emph{composing space} $Y$  is a compact 3-manifold $W\times S^1$ having $n+1$ boundary tori $\Sigma_j=\partial D_j \times S^1$ $(1\leq j \leq n)$ and  $\Sigma = \partial D^2 \times S^1$.
		Letting $\mu= \{pt\} \times S^1$ and $\lambda_j = \partial D_j \times \{ pt \}$, we have
		\begin{equation*}
			\pi_1(Y)=\langle \mu,\lambda_1, \ldots, \lambda_{n} \,|\, [\mu,\lambda_1]=\cdots= [\mu,\lambda_{n}]=1 \rangle.			
		\end{equation*}
		One can check that $H_0(Y;\Rbb) \simeq \Rbb$ is generated by a point $p \in Y$, $H_1(Y;\Rbb)\simeq \Rbb^{n+1}$ is generated by $\mu,\lambda_1,\ldots,\lambda_n$, and $H_2(Y;\Rbb) \simeq \Rbb^n$ is generated by $\Sigma_1,\ldots,\Sigma_n$. 
		We choose the orientation  $o_Y$ of $H_\ast(Y;\Rbb)$ induced from a basis $h_Y=(p,\mu,\lambda_1,\ldots,\lambda_n, \Sigma_1,\ldots,\Sigma_n)$ of $H_\ast(Y;\Rbb)$.
		Here we orient $\mu$, $\lambda_j$, and $\Sigma_j$  as in   Example \ref{exam:torus} and Stokes' theorem.

		Let  $\rho : \pi_1(Y) \rightarrow \G$ be a representation with $\tr \rho(\mu) \neq \pm 2$ and $\tr \rho(\lambda_j) \neq \pm 2$ for some $1 \leq j\leq n$. Since $\mu$ commutes with all $\lambda_j$'s, we have up to conjugation 
		\begin{equation}\label{eqn:diag}
			\rho(\mu) =\begin{pmatrix} m & 0 \\ 0 & m^{-1} \end{pmatrix}, \quad \rho(\lambda_j) = \begin{pmatrix} l_j & 0 \\ 0 & l_j^{-1}\end{pmatrix} 
		\end{equation}
		for some $m \neq \pm1 $ and $l_j \in \Cbb^\ast$. Note that there is no relation among $m$, $l_1,\ldots,l_n$.
%		and  $l_j \neq \pm1$ for some $1 \leq j \leq n$.
		\begin{proposition} \label{lem:Yh}
			We have
			\begin{equation}
				\dim H^i(Y;\g_\rho) = \left \{ 
				\begin{array}{ll}
					1 &i=0,\\
					n+1 & i=1,\\
					n & i=2,\\
					0 & otherwise.
				\end{array}
				\right.
			\end{equation}
		\end{proposition}
		\begin{proof}
			We first compute the twisted cohomology of $W$. 	
			Since $W$ retracts to the wedge sum $V$ of $n$ circles $\lambda_1, \ldots,\lambda_n$ (with the basepoint $p$), we may consider $V$ instead of $W$:
			\[ \g \simeq C^0(V;\g_\rho) \overset{\delta^0}{\rightarrow} C^1(V;\g_\rho) \simeq \g^n, \quad \delta^0 = \begin{pmatrix}
				\Ad(\lambda_1)-I_3 \\
				 \vdots  \\
				 \Ad(\lambda_n)-I_3 
				\end{pmatrix}. \]
			From the equation \eqref{eqn:diag} with the fact that $\tr \rho(\lambda_j) \neq \pm 2$  for some $1 \leq j \leq n$, we have
			\begin{equation}
				\dim H^i(W;\g_\rho) = \dim H^i(V;\g_\rho) = \left \{ 
				\begin{array}{ll}
					1 &i=0,\\
					3n-2 & i=1,\\
					0 & otherwise.
				\end{array}
				\right.
			\end{equation}  
			Without loss of generality, we assume that $l_1  \neq  \pm 1$ and choose a basis $\mathbf{h}^i_W$ of $H^i(W;\g_\rho)$ as   
			\[\mathbf{h}^0_W = p^{(2)}, \quad \mathbf{h}^1_W = (\underbrace{\lambda_1^{(2)},\ldots,\lambda_n^{(2)}}_{n},\ \underbrace{\lambda_2^{(1)},\ldots,\lambda_n^{(1)}}_{n-1},\ \underbrace{\lambda_2^{(3)},\ldots,\lambda_n^{(3)}}_{n-1}).\]
			Here we choose a lift of  $p$ arbitrarily and determines the lifts of  other cells accordingly. Recall Section \ref{sec:cw} that the notations $p^{(k)}$ and $\lambda_j^{(k)}$ make sense after we fix lifts of $p$ and $\lambda_j$.

			We decompose $Y$ into two copies $Y_1$ and $Y_2$ of $W \times I$ where $I$ is an interval. 
			It is clear that  both $Y_1$ and $Y_2$ retract to $W$ and $Y_1 \cap Y_2 = W \sqcup W$. From the short exact sequence 
		 \begin{equation} \label{eqn:sesY}
	0 \rightarrow C^\ast(Y;\g_\rho) \rightarrow C^\ast(Y_1;\g_\rho) \oplus C^\ast(Y_2;\g_\rho) \rightarrow  C^\ast(W;\g_\rho) \oplus C^\ast(W;\g_\rho)\rightarrow 0,
\end{equation}
			 we obtain
			\begin{align}
				\label{eqn:les}
				\mathcal{H}:  0 & \longrightarrow H^0(Y;\g_\rho) \overset{f_0}{\longrightarrow}   H^0(W;\g_\rho) \oplus H^0(W;\g_\rho) \overset{g_0}{\longrightarrow}  H^0(W;\g_\rho) \oplus H^0(W;\g_\rho) \\
				&  \overset{d_0}{\longrightarrow} H^1(Y;\g_\rho) \overset{f_1}{\longrightarrow}  H^1(W;\g_\rho) \oplus H^1(W;\g_\rho)  \overset{g_1}{\longrightarrow} H^1(W;\g_\rho)  \oplus H^1(W;\g_\rho)  \overset{d_1}{\longrightarrow} H^2(Y;\g_\rho) \longrightarrow 0. \nonumber
			\end{align}
			Fixing identifications $H^0(W;\g_\rho) \simeq \Cbb$ and $H^1(W;\g_\rho) \simeq \Cbb^{3n-2}$ with respect to $\mathbf{h}_W$,
			 the matrix expressions of $g_0$ and $g_1$ in the sequence \eqref{eqn:les} are given by
			\begin{equation} \label{eqn:gs}
				g_0 = \begin{pmatrix} 
					1 & -1\\
					-1 & 1
				\end{pmatrix}, \quad
				g_1= \begin{pmatrix}
					I_{3n-2} & \rvline & \begin{matrix} 
						-I_n & 0 & 0 \\
						0 & -m^2 I_{n-1} & 0 \\ 
						0& 0 &-m^{-2} I_{n-1}
					\end{matrix}     \\
					\hline
					-I_{3n-2}& \rvline  & I_{3n-2} 
				\end{pmatrix}
			\end{equation}
			where $I_k$ is the identity matrix of size $k$.
			In particular,  $\mathrm{Ker}\,g_0$ is generated by  $\mathfrak{e}_1 + \mathfrak{e}_2$ and $\mathrm{Ker}\,g_1 $ is generated by $\mathfrak{e}_1 + \mathfrak{e}_{3n-1},\ldots, \mathfrak{e}_n + \mathfrak{e}_{4n-1}$. Here $\mathfrak{e}_k$ is a unit vector whose coordinates are all zero, except one at the $k$-th coordinate.
			It follows that $\dim \mathrm{Im}\, g_0 =1 $, $\dim \mathrm{Im}\, g_1 = 5n-4$, and 
			\begin{align*}
				\dim H^0(Y;\g_\rho) &= \dim \mathrm{Im}\,f_0= \dim \mathrm{Ker}\,g_0=1, \\
				\dim H^2(Y;\g_\rho) &= 2 \dim H^1(W;\g_\rho) - \dim \mathrm{Im}\,g_1=n.
			\end{align*}
			Also,  we have $\dim H^1(Y;\g_\rho) = n+1$ since the Euler characteristic of $Y$ is zero.
		\end{proof}

		It is geometrically natural to choose a basis $\mathbf{h}^i_Y$ of $H^i(Y;\g_\rho)$ as
		\begin{equation} \label{eqn:hY}
			\mathbf{h}^0_Y=  p^{(2)},\  \mathbf{h}^1_Y = (\mu^{(2)},\lambda_1^{(2)},\ldots, \lambda_n^{(2)}),\
			\mathbf{h}^2_Y= (\Sigma_1^{(2)},\ldots,\Sigma_n^{(2)}).
		\end{equation}
	Alternatively, we may describe the basis $\mathbf{h}_Y$ as follows (as in Example \ref{exam:torus}).
	Let $P=\frac{1}{8}e_2 \in H^0(Y;\g_\rho)$ and consider isomorphisms
	\begin{align*}
		& \psi^0 : H^0(Y;\g_\rho) \rightarrow \Cbb,\ \alpha \mapsto \langle \alpha(\widetilde{p}), P \rangle, \\
		& \psi^1 : H^1(Y;\g_\rho) \rightarrow \Cbb^{n+1}, \ \alpha \mapsto \left(\langle \alpha(\widetilde{\mu}), P \rangle, \langle \alpha(\widetilde{\lambda}_1), P \rangle,\ldots,\langle \alpha(\widetilde{\lambda}_n), P \rangle\right),\\
		& \psi^2 : H^2(Y;\g_\rho) \rightarrow \Cbb^n,  \ \alpha \mapsto \left(\langle \alpha(\widetilde{\Sigma}_1), P \rangle,\ldots,\langle \alpha(\widetilde{\Sigma}_n), P \rangle\right).
	\end{align*}  
Then  the basis $\mathbf{h}^i_Y$ maps to the standard basis of $\Cbb$, $\Cbb^{n+1}$, or  $\Cbb^n$ under $\psi^i$ accordingly.

		\begin{proposition}  \label{lem:YTor} $\tau(Y;\rho,\mathbf{h}_Y,o_Y)=(-1)^{n-1}(m-m^{-1})^{2n-2}$
		\end{proposition}
		\begin{proof}
			 Recall  that $Y$ decomposes into two copies $Y_1$ and $Y_2$ of $W \times I$ with $Y_1 \cap Y_2 =W \sqcup W$  and  that $W$ retracts to $V$, the wedge sum of $n$ circles $\lambda_1,\ldots,\lambda_n$ with the base point $p$.

			 We construct $V\times I$ from two copies of $V$ (regarding them as $V \times \partial I$) by adding cells $p\times I$, $\lambda_1\times I,\ldots,\lambda_n\times I$.
			 Choose cell orders of $V$, $V \times I,$ and $V \times S^1$ as
			 \begin{itemize}
			 	\item $c_V=(p,\lambda_1,\ldots,\lambda_n)$,
			 	\item $c_{V \times I}=(c_V, c_V, c_{\widetilde{V}})$ where $c_{\widetilde{V}}=(p\times I$, $\lambda_1\times I,\ldots,\lambda_n\times I)$,
			 	\item $c_{V \times S^1}=(c_V,c_V,c_{\widetilde{V}},c_{\widetilde{V}})$.
			 \end{itemize} 
			 Then the basis transition between $(c_{V\times I}, c_{V\times I})$ and $(c_V,c_V,c_{V\times S^1})$ is an even permutation.
			 On the other hand, for $h_{V\times S^1}=(p,\mu,\lambda_1,\ldots,\lambda_n,\Sigma_1,\ldots,\Sigma_n)\,(=h_Y)$
			  a straightforward computation shows that 
			 \begin{align}
			 &	\mathrm{Tor}(C_\ast(V\times S^1;\Rbb),c_{V\times S^1},h_{V\times S^1}) \label{eqn:torsign} \\
			 &	=(-1)^{|C_\ast(V\times S^1;\Rbb)|}\,\det \begin{pmatrix}
			 		I_n & 0 \\
			 		I_n &  I_n
			 	\end{pmatrix}^{-1} \det \begin{pmatrix}
			 		-I_n & \rvline & 0 & \rvline & I_n & \rvline & 0 \\
			 		\hline
			 		I_n  & \rvline & 0  & \rvline & 0  & \rvline & 0 \\
			 		\hline
			 		0  & \rvline & \begin{matrix} 
			 			1 \\ 
			 			1
			 		\end{matrix}  & \rvline & 0  & \rvline &
			 		\begin{matrix} 
			 			0 \\ 
			 			1
			 		\end{matrix}
			 	\end{pmatrix}
			 	\det \begin{pmatrix}
			 		1 & 1 \\
			 		-1 & 0 
			 	\end{pmatrix}^{-1}=1. \nonumber
			 \end{align}
			Note that $|C_\ast(V\times S^1;\Rbb)|$ is obviously even.
			
			We choose any triangulation of $Y$ and cell orders $c_Y$, $c_{Y_i}$, and $c_W$ according to $c_{V \times S^1}$, $c_{V\times I}$, and $c_V$, respectively.
			Applying the formula~\eqref{eqn:gluing} to the short exact sequence \eqref{eqn:sesY},
	 	we obtain 
	 	\begin{equation}	 		
	 		 1 = (-1)^{v+u}\, \mathrm{Tor}(C^\ast(Y;\g_\rho),\mathbf{c}_Y,\mathbf{h}_Y)\, 
	 		\mathrm{tor}(\mathcal{H}) \nonumber
	 	\end{equation}
%				\begin{equation}
%				\mathrm{Tor}(C^\ast(Y;\g_\rho),\mathbf{c}_Y, \mathbf{h}_Y)\mathrm{Tor}(C^\ast(Y;\g_\rho),\mathbf{c}_Y, \mathbf{h}_Y) =(-1)^{n+1} \mathrm{tor}(\mathcal{H})^{-1}
%			\end{equation}		
%			\begin{equation*}
%					\tau(Y;\rho,\mathbf{h}_Y,o_Y)= \epsilon(o_Y)\, \mathrm{tor}(\mathcal{H})^{-1}
%			\end{equation*}
			after canceling out the torsion terms for $W \simeq Y_i$.
			Here $\mathrm{tor}(\mathcal{H})$ is the  Reidemeister torsion of the long exact sequence \eqref{eqn:les} with respect to $\mathbf{h}_Y$ and $\mathbf{h}_W$.
			Note that the basis transition between $(\mathbf{c}_{Y_1},\mathbf{c}_{Y_2})$ and $(\mathbf{c}_Y,\mathbf{c}_W,\mathbf{c}_W)$ is an even permutation.
			One easily checks from the definitions \eqref{eqn:v} and \eqref{eqn:u} that $v\equiv0$ and $u\equiv \sum_i \beta_i(C^\ast(Y;\g_\rho)) \equiv n-1$  in modulo 2.
			To simplify notations, we rewrite the sequence \eqref{eqn:les} as 
%			\begin{align}
%				0 & \longrightarrow \mathcal{H}^0 \overset{f_0}{\longrightarrow}   \mathcal{H}^1  \overset{g_0}{\longrightarrow}  \mathcal{H}^2
%				\overset{d_0}{\longrightarrow} \mathcal{H}^3 \overset{f_1}{\longrightarrow} \mathcal{H}^4 \overset{g_1}{\longrightarrow} \mathcal{H}^5  \overset{d_1}{\longrightarrow} \mathcal{H}^6 \longrightarrow 0. 
%			\end{align}
			\begin{center}
				\begin{tikzcd}[column sep=0.6cm]
					0 \arrow[r]  & \mathcal{H}^0 \arrow[r,"f_0"] \arrow[d, "\simeq"] & \mathcal{H}^1 \arrow[r,"g_0"]  \arrow[d, "\simeq"]& \mathcal{H}^2 \arrow[r,"d_0"] \arrow[d, "\simeq"] &  \mathcal{H}^3 \arrow[r,"f_1"]  \arrow[d, "\simeq"] & \mathcal{H}^4 \arrow[r,"g_1"]  \arrow[d, "\simeq"] & \mathcal{H}^5  \arrow[r,"d_1"] \arrow[d, "\simeq"] & \mathcal{H}^6 \arrow [r] \arrow[d, "\simeq"] & 0\\ 
					 & \Cbb \arrow[r,"f_0"]& \Cbb^2 \arrow[r,"g_0"] & \Cbb^2 \arrow[r,"d_0"] & \Cbb^{n+1} \arrow[r,"f_1"] & \Cbb^{6n-4} \arrow[r,"g_1"]&  \Cbb^{6n-4} \arrow[r,"d_1"]& \Cbb^n & 
				 \end{tikzcd}
			\end{center}
			where the first and second rows are identified with respect to  $\mathbf{h}_Y$ and $\mathbf{h}_W$.
			We choose a tuple $b^i$ of vectors in $\mathcal{H}^i$ as 			
			\begin{align*} 
				b^0&=\mathfrak{e}_1, \  b^1=\mathfrak{e}_1,\ b^2=\mathfrak{e}_1,\ b^3=(\mathfrak{e}_2, \mathfrak{e}_3,\ldots, \mathfrak{e}_{n+1}),\\
				b^4&=(\mathfrak{e}_{n+1},\mathfrak{e}_{n+2}, \ldots,\mathfrak{e}_{6n-4}),\ b^5 = (\mathfrak{e}_1, \mathfrak{e}_2,\ldots,\mathfrak{e}_n), \ b^6 = \emptyset
			\end{align*}
			where $\mathfrak{e}_k$ is a unit vector whose coordinates are all zero, except one at the $k$-th coordinate. Then the basis transition matrix $A_i$ at $\mathcal{H}^i$ (see Section \ref{sec:cw}) is given by 
			\begin{align*}
			&	A_0=I_1,\ 
				A_1=\begin{pmatrix}
					1 & 1 \\
					1 & 0	
				\end{pmatrix}, \
				A_2=\begin{pmatrix}
					1 & 1 \\
					-1 & 0	
				\end{pmatrix},\ A_3= I_{n+1} \\
			& A_4=
			\begin{pmatrix}
				\begin{matrix}
				I_{n} & 0 \\
				0 & I_{2n-2}
				\end{matrix} & \rvline & 0 \\
				\hline
				\begin{matrix}
					I_n & 0 \\
					0 & 0
				\end{matrix} & \rvline &  I_{3n-2} 
			\end{pmatrix}, \ 
			A_5=
			\begin{pmatrix}
			\begin{matrix}
				0 \\
				I_{2n-2}
			\end{matrix}   & \rvline & 
			\begin{matrix}
				-I_n & 0& 0\\
				0& -m^2 I_{n-1} &0 \\
				0&0 &  -m^{-2} I_{n-1} 
			\end{matrix} & \rvline & \begin{matrix} I_n \\ 0 \end{matrix} \\
			\hline
			\begin{matrix}
				0 \\
				-I_{2n-2}
			\end{matrix}   & \rvline & I_{3n-2} & \rvline &  0
		\end{pmatrix}, \ A_6= I_n.
			\end{align*}
					%		\begin{align*} 
%		b^0&=p^{(2)}, \  b^1=(p^{(2)},0),\ b^2=(p^{(2)},0),\ b^3=(\lambda_1^{(2)}, \cdots,\lambda_n^{(2)}),\\
%		b^4&=(\underbrace{(\lambda_2^{(1)},0),\cdots,(\lambda_n^{(1)},0)}_{n-1},\ \underbrace{(\lambda_2^{(3)},0),\cdots,(\lambda_n^{(3)},0)}_{n-1},\\
%		& \quad \quad \underbrace{(0,\lambda_1^{(2)}),\cdots,(0,\lambda_n^{(2)})}_{n},\ \underbrace{(0,\lambda_2^{(1)}),\cdots,(0,\lambda_n^{(1)})}_{n-1},\ \underbrace{(0,\lambda_2^{(3)}),\cdots,(0,\lambda_n^{(3)})}_{n-1}),\\
%		b^5 &= ((\lambda_1^{(2)},0),\cdots,(\lambda_n^{(2)},0)),\  b^6 = \emptyset,
%	\end{align*}
%		(note that both $b^1$ and $b^2$ are singletons of $(p^{(2)},0) \in H^0(W;\g_\rho) \oplus H^0(W;\g_\rho)$).
	Here we used the equation \eqref{eqn:gs} with the fact that $f_0(\mathfrak{e}_1) = \mathfrak{e}_1+\mathfrak{e}_2$,  $f_1(\mathfrak{e}_1) =0$,  $f_1(\mathfrak{e}_{j+1}) = \mathfrak{e}_j+\mathfrak{e}_{3n-2+j}$, $d_0(\mathfrak{e}_1) = \mathfrak{e}_1$, $d_1(\mathfrak{e}_{j}) = \mathfrak{e}_j $ for  $1 \leq j \leq n$.			
%			\begin{align*}
%				& f_0(p^{(2)}) = (p^{(2)}, p^{(2)}),\quad f_1(\lambda_i^{(2)}) = (\lambda_i^{(2)},\lambda_i^{(2)}),\quad  f_1(\mu^{(2)}) =(0,0), \\
%				& d_0(p^{(2)},0) = \mu^{(2)},\quad  d_1(\lambda_i^{(2)},0) = \Sigma_i^{(2)}, 			\end{align*}
%		one easily checks that contribution to $\mathrm{tor}(\mathcal{H})$ other than $\pm1$ only comes from $g_1 : \mathcal{H}^4 \rightarrow \mathcal{H}^5$. Precisely, we have
			It follows that 
			\begin{align*}
				\mathrm{tor}(\mathcal{H}) &= -\det
A_5^{-1} \\[-10pt]
				&= (-1)^{n-1} \det     \begin{pmatrix}
					I_{2n-2}
					& \rvline & 
					\begin{matrix}
						0 & -m^2 I_{n-1} &0 \\
						0  & 0& - m^{-2}I_{n-1} 
					\end{matrix}  \\
					\hline
					\begin{matrix}
						0 \\
						-I_{2n-2}
					\end{matrix}   & \rvline & I_{3n-2} 
				\end{pmatrix}^{-1} \\
				&= (-1)^{n-1}\det \begin{pmatrix} 
					(1-m^2)I_{n-1} & 0 \\
					0 & (1-m^{-2})I_{n-1}
				\end{pmatrix}^{-1}\\
				&= (m-m^{-1})^{2-2n}.
			\end{align*}
			Note that the third equation follows from the determinant formula for a block matrix. 
			We conclude that
	 	\begin{equation}	 		
			 \mathrm{Tor}(C^\ast(Y;\g_\rho),\mathbf{c}_Y,\mathbf{h}_Y)=
			(-1)^{n+1}(m-m^{-1})^{2n-2}. \nonumber
		\end{equation}
		This completes the proof, since we have $\epsilon(o_Y) =1$ from the equation \eqref{eqn:torsign}.
		\end{proof}
		\begin{remark}  We have $\tau(Y;\rho,\mathbf{h}_Y,o_Y)=1$ for $n=1$. It agrees with the computation given in Example \ref{exam:torus} since $Y$ retracts to a 2-torus when $n=1$.
		\end{remark}

		\subsection{The knot exterior of the connected sum} \label{sec:gluing}
		The composing space $Y$ has $n+1$ boundary tori  $\Sigma_1,\ldots,\Sigma_n$, and $\Sigma$.
		For $1 \leq j \leq n$ we glue the knot exterior $M_j$ of  $K_j \subset S^3$ to $Y$ by using a homeomorphism $\partial M_j\rightarrow  \Sigma_j$ that maps the meridian and canonical longitude of $K_j$ to $\mu$ and $\lambda_j$, respectively.
		The resulting manifold $M$ is a compact 3-manifold with $\partial M = \Sigma$ and is the knot exterior of the connected sum of $K_1,\ldots,K_n$.  
 		We refer to \cite[IX.21--22]{jaco1980lectures} for details.
		We choose the orientation $o_M$ of $H_\ast(M;\Rbb)$ as in Section \ref{sec:knot}, i.e., the one induced from the basis  $h_M=(pt, \mu)$ of $H_\ast(M;\Rbb)$. 
%		From the Mayer-Vietoris sequence for $\sqcup_{j=1}^n M_j$ and $Y$, we obtain a short exact sequence
%\[ 0\rightarrow \bigoplus_{j=1}^n H_i(\Sigma
%_j;\Rbb) \rightarrow \bigoplus_{j=1}^n H_i(M
%_j;\Rbb) \oplus H_i(Y;\Rbb)  \rightarrow H_i(\mathcal{M};\Rbb)  \rightarrow 0\]
%for $i=0,1,2$. It follows that $\dim H_i(\mathcal{M};\Rbb)=1$ for $i=0,1$ and $\dim H_i(\mathcal{M};\Rbb)=0$ otherwise. We let $o_\mathcal{M}$ be an orientation of $H_\ast(\mathcal{M};\Rbb)$ induced from $(pt,\mu)$. 

		Let $\rho: \pi_1(M)\rightarrow \G$ be an irreducible representation. 
		We  denote by $m$ and $l_j$ eigenvalues of $\rho(\mu)$ and $\rho(\lambda_j)$ respectively as in the equation \eqref{eqn:diag}. For simplicity we assume that 
		\begin{equation} \label{eqn:assum}
			m \neq \pm1 \textrm{ and } \Delta_{K_j} (m^2) \neq 0 \textrm{ for all } 1 \leq  j \leq n
		\end{equation}
		where $\Delta_{K_j}$ is the Alexander polynomial of $K_j$.
		It follows that each restriction $\rho_j:\pi_1(M_j)\rightarrow \G$ of $\rho$ is either irreducible or abelian. 
		We further assume that if $\rho_j$ is irreducible, then
		\begin{equation}\label{eqn:assum2}
			 l_j \neq \pm 1 \textrm{ and } \rho_j \textrm{ is } \mu \textrm{-regular.}	
		\end{equation}
		Without loss of generality, we assume that $\rho_1,\ldots, \rho_k$ are abelian and $\rho_{k+1},\ldots,\rho_n$ are irreducible where $k$ should be less than $n$, otherwise $\rho$ becomes abelian. In particular, $l_j \neq \pm 1$ for some $1 \leq j \leq n$.
	
%		\begin{remark}
%		Contrast to Section \ref{sec:connector}, as we glued the knot exterior $M_i$ to $Y$,
%the variables $m$ and $l_i$ are no longer free. In particular, we have
%\[A_i(l_i,m)=0\]
%for $1 \leq i \leq n$ where $A_i$ is the A-polynomial of $K_i$ \cite{cooper1994plane}.			
%		\end{remark}

		\begin{proposition} \label{prop:mh} We have
			\begin{equation}
				\dim H^i(M;\g_\rho) = \left \{ 
				\begin{array}{ll}
					n-k & i=1,2,\\
					0 & otherwise.
				\end{array}
				\right.
			\end{equation}
		\end{proposition}

		\begin{proof} 
		From the short exact sequence
		\begin{equation} \label{eqn:ese3}
				0 \rightarrow C^\ast(M;\g_\rho) \rightarrow \bigoplus_{j=1}^n C^\ast(M_j;\g_\rho) \oplus C^\ast(Y;\g_\rho) \rightarrow \bigoplus_{j=1}^n C^\ast(\Sigma_j;\g_\rho) \rightarrow 0,
			\end{equation}
			we have
			\begin{align}
				\mathcal{G}: 0 
				& \longrightarrow  H^0(M;\g_\rho) \overset{F_0}{\longrightarrow} \bigoplus_{j=1}^n H^0(M_j;\g_\rho) \oplus  H^0(Y;\g_\rho) \overset{G_0}{\longrightarrow} \bigoplus_{j=1}^n H^0(\Sigma_j;\g_\rho)  \label{eqn:les2}\\ 
				& \overset{D_0}{\longrightarrow} H^1(M;\g_\rho) \overset{F_1}{\longrightarrow}  \bigoplus_{j=1}^n H^1(M
				_j;\g_\rho) \oplus H^1(Y;\g_\rho)  \overset{G_1}{\longrightarrow} \bigoplus_{j=1}^n H^1(\Sigma
				_j;\g_\rho) \nonumber \\
				&  \overset{D_1}{\longrightarrow} H^2(M;\g_\rho) \overset{F_2}{\longrightarrow}  \bigoplus_{j=1}^n H^2(M
				_j;\g_\rho) \oplus H^2(Y;\g_\rho)  \overset{G_2}{\longrightarrow} \bigoplus_{j=1}^n H^2(\Sigma
				_j;\g_\rho) \overset{D_2}{\longrightarrow} H^3(M;\g_\rho)\longrightarrow0. \nonumber
			\end{align}
%		Recall 
%		 that we have the bases $\mathbf{h}_{M_j}$, $\mathbf{h}_Y$,   and $\mathbf{h}_{\Sigma_j}$ of $H^\ast(M_j)$, $H^\ast(Y)$, and  $H^\ast(\Sigma_j)$ given in Example \ref{exam:torus} and Sections \ref{sec:knot} and \ref{sec:connector}, respectively.
%		 
		With respect to the bases $\mathbf{h}_{\Sigma_j}$, $\mathbf{h}_{M_j}$, and $\mathbf{h}_Y$ given in Example \ref{exam:torus} and Sections \ref{sec:knot} and \ref{sec:connector}, the map $G_0$ in the sequence \eqref{eqn:les2} agrees with		
			\begin{equation}\label{eqn:g0}
				G_0: \Cbb^{k+1} \rightarrow \Cbb^n,\ (x_1,\ldots,x_k,y) \mapsto (x_1-y,\cdots,x_k-y, \, \underbrace{-y, \ldots,-y}_{n-k}).
			\end{equation}
		It follows that $\dim \mathrm{Im}\,G_0 =\dim \mathrm{Ker}\,D_0=k+1$ and $\dim \mathrm{Im}\,D_0= n-k-1$.
		Also, the matrix expression of $G_1 : \Cbb^{2n+1} \rightarrow \Cbb^{2n}$ is given by
			\begin{equation} \label{eqn:g1}
				G_1 = \begin{pmatrix}
				1 & \cdots & 0 & \rvline & -1 & \rvline & 0 & \cdots & 0 \\
				\kappa_1&   & 0 & \rvline & 0 & \rvline  & -1 &  & 0\\
				\vdots & \ddots& \vdots & \rvline &  \vdots & \rvline & \vdots & \ddots & \vdots \\
				0 &  & 1 &  \rvline & -1 &\rvline & 0 & & 0  \\
				0 & \cdots & \kappa_n & \rvline & 0 & \rvline& 0 & \cdots & -1 
			\end{pmatrix}
		\end{equation}
		  where $\kappa_j = \langle \mathbf{h}^1_{M_j}(\widetilde{\lambda_j}),P\rangle$.
		  Since we obtain an invertible matrix (of size $2n$) from $G_1$  by deleting the $(n+1)$-st column, we have $\dim \mathrm{Im}\,G_1 = 2n$  and $\dim \mathrm{Ker}\, G_1 =
		  \dim \mathrm{Im}\,F_1=1$. 
%		  Explicitly, $\mathrm{Ker}\, g_1$ is generated by
%			\begin{equation}
%				( \underbrace{1,\cdots,1}_{n+1},  \, \kappa_1,\cdots, \kappa_n )^T.
%			\end{equation}
		  It follows that $\dim H^1(M;\g_\rho) = \dim \mathrm{Im}\, F_1 + \dim \mathrm{Im}\, D_0=n-k$. 
		  On the other hand, $G_2$ is surjective, since the restriction map $H^2(Y;\g_\rho) \rightarrow \oplus_{j=1}^n H^2(\Sigma_j;\g_\rho)$  is an isomorphism (see the equation~\eqref{eqn:hY}). 
		  It follows that $\dim H^3(M;\g_\rho)=0$ and   $\dim H^2(M;\g_\rho)=n-k$ since the Euler characteristic of $M$ is zero.
%			Recall that we have the bases $\mathbf{h}^2_{M_i}$, $\mathbf{h}^2_Y$,   and $\mathbf{h}^2_{\Sigma_i}$ of $H^2(M_i)$, $H^2(Y)$, and  $H^2(\Sigma_i)$, respectively, for $1 \leq i \leq n$. With respect to these bases, we have 
%			\[g_2 = \begin{pmatrix} 
%				I_n &  -I_n
%			\end{pmatrix}.\]
%			Since $g_2$ is clearly surjective, $d_2$ is a zero map and thus $H^3(M)=0$. 	
		\end{proof}
	
		We let $P=\frac{1}{8}e_2 \in H^0(Y;\g_\rho)$ and define maps 
		\begin{align*}
		& \psi^1 : H^1(M;\g_\rho) \rightarrow \Cbb, \ \alpha \mapsto\langle \alpha(\widetilde{\mu}), P \rangle,\\
		& \psi^2 : H^2(M;\g_\rho) \rightarrow \Cbb^{n-k}, \ \alpha \mapsto \left(\langle \alpha(\widetilde{\Sigma}_{k+1}), P \rangle,\ldots,\langle \alpha(\widetilde{\Sigma}_{{n}}), P \rangle \right).
	\end{align*}  	
		\begin{lemma} \label{lem:isoms} $\psi^1$ induces an isomorphism  $H^1(M;\g_\rho)/\mathrm{Im}\,D_0 \rightarrow \Cbb$ and  $\psi^2$ is an isomorphism.
		\end{lemma}
		\begin{proof} It is clear that $\psi^1$ is compatible with the isomorphism $H^1(M_j;\g_\rho)\rightarrow \Cbb$, $\alpha \mapsto \langle \alpha(\widetilde{\mu}),P \rangle$ for  $1 \leq j \leq n$.
			In particular, $\psi^1$ is surjective.	
			On the other hand, it follows from the sequence \eqref{eqn:les2} that an element of $\mathrm{Im}\, D_0$ maps to the trivial element in $H^1(M_j;\g_\rho)$ under the restriction map $H^1(M;\g_\rho) \rightarrow H^1(M_j;\g_\rho)$. Therefore, $\psi^1$ induces a map $H^1(M;\g_\rho)/ \mathrm{Im}\,D_0 \rightarrow \Cbb$ which is an isomorphism since $\dim H^1(M;\g_\rho)=n-k$ and $\dim \mathrm{Im}\,D_0=n-k-1$.
			The second claim that $\psi^2$ is an isomorphism  is obvious from the sequence \eqref{eqn:les2}.
		\end{proof}
	
		Recall that the basis $\mathbf{h}_{\Sigma_j}^0$ of $H^0(\Sigma_j;\g_\rho)$ gives us an isomorphism $\oplus_{j=1}^n H^0(\Sigma_j ;\g_\rho) \simeq \Cbb^n$.
		Denoting by $(\mathfrak{e}_1,\ldots,\mathfrak{e}_n)$ the standard basis  of $\Cbb^n$,
%		Recall that $H^0(\Sigma_i;\g_\rho)\simeq \Cbb$ is generated by $\mathbf{h}_{\Sigma_i}^0$ and thus a natural basis of $\oplus_{i=1}^n H^0(\Sigma_i;\g_\rho) \simeq \Cbb^n$ is given by
%		\[ \left((\mathbf{h}^0_{\Sigma_1},0,\cdots,0,0),\, (0,\mathbf{h}^0_{\Sigma_2},\cdots,0,0), \cdots, (0,0,\cdots,0,\mathbf{h}^0_{\Sigma_n}) \right)\]
	we choose a basis of $\mathrm{Im}\, D_0$ as
	\begin{equation*}
		 \left( D_0(\mathfrak{e}_{k+1}),\ldots, D_0(\mathfrak{e}_{n-1}) \right).
	\end{equation*}
	Note that the equation \eqref{eqn:g0} implies that the above tuple is indeed a basis of $\mathrm{Im}\,D_0$.
	We then extend it to a basis $\mathbf{h}^1_M$ of $H^1(M;\g_\rho)$ by adding an element $\xi$  at the first position which maps to the standard basis of $\Cbb$ under $\psi^1$: 
	\[\mathbf{h}^1_M =(\xi, D_0(\mathfrak{e}_{k+1}),\ldots, D_0(\mathfrak{e}_{n-1})).\] 	We also choose  a basis $\mathbf{h}^2_M$ of $H^2(M;\g_\rho)$ by the pre-image of the standard basis of $\Cbb^{n-k}$ under $\psi^2$ and define the \emph{adjoint Reidemeister torsion} by
	\begin{equation} \label{eqn:defn}
		\tau_\mu(M;\rho)=\tau(M;\rho, \mathbf{h}_M ,o_M).
	\end{equation}
		Note that the above definition reduces to the definition  \eqref{eqn:torog} of a knot exterior when $n=1$.
		\begin{lemma}
			The definition \eqref{eqn:defn} does not depend on the choice of $P \in H^0(Y;\g_\rho)$ and the order of indices of $\Sigma_{k+1},\cdots, \Sigma_n$.
		\end{lemma}
		\begin{proof} If we replace $P$ by $cP$ for some $c \in \Cbb^\ast$, then the basis transition matrices for $H^i(M;\g_\rho)$ is $\frac{1}{c} I_{n-k}$  for both $i=1,2$ and thus $\tau_\mu(M;\rho)$ does not change. 
		If we exchange two indices other than $n$, then the basis transition is clearly an odd permutation for both  $H^1(M;\g_\rho)$ and  $H^2(M;\g_\rho)$. Therefore $\tau_\mu(M;\rho)$ does not change.
		If we  exchange the index $n$ with another one, then the basis transition  for $H^2(M;\g_\rho)$ is an odd permutation.
		On the other hand, since $\mathfrak{e}_{k+1} +\cdots +\mathfrak{e}_n \in \mathrm{Im}\, G_0 = \mathrm{Ker}\,D_0$ (see the equation \eqref{eqn:g0}), we have $D_0(\mathfrak{e}_n)=-D_0(\mathfrak{e}_{k+1}) - \cdots - D_0(\mathfrak{e}_{n-1})$.
		It follows that the basis transition matrix for $H^1(M;\g_\rho)$ has determinant $-1$ and thus  $\tau_\mu(M;\rho)$ does not change.
		\end{proof}
		\begin{theorem} 		\label{lem:torproduct}
		$\tau_\mu(M;\rho) =  (m-m^{-1})^{2n-2}\,\tau_\mu(M_1;\rho_1)\cdots \tau_{\mu} (M_n;\rho_n).$
		\end{theorem}
		\begin{proof}
			Choose any triangulation of $M$ with any cell order $c_M$.
			We denote by $c_Y$ (resp., $c_{M_j}$ and $c_{\Sigma_j}$) the cell order 		restricted to  $Y$ (resp., $M_j$ and $\Sigma_j$). 
			Note that the Euler characteristics of $M$, $M_j$, $Y$, $\Sigma_j$ are even. 
			It follows that we may assume that the number of  $i$-dimensional cells in each $M$, $M_j$, $Y$, and $\Sigma_j$ is even by applying the barycentric subdivision once.
			Let $e=1$ (resp., $-1$) if the basis transition between 
			$(c_{\Sigma_1},\ldots,c_{\Sigma_n},c_M)$ and $(c_{M_1},\ldots,c_{M_n},c_Y)$ is an even (resp., odd) permutation.

			Applying the formula \eqref{eqn:gluing} to the short exact sequence \eqref{eqn:ese3},
			 we obtain
 			\begin{align}
			 	&e \cdot \prod_{j=1}^n\mathrm{Tor}(C^\ast(M_j;\g_\rho), \mathbf{c}_{M_j}, \mathbf{h}_{M_j}) \cdot \mathrm{Tor}(C^\ast(Y;\g_\rho), \mathbf{c}_{Y}, \mathbf{h}_{Y}) \nonumber \\
			 	&=(-1)^{v+u} \mathrm{Tor}(C^\ast(M;\g_\rho), \mathbf{c}_{M}, \mathbf{h}_{M}) \cdot \prod_{j=1}^n\mathrm{Tor}(C^\ast(\Sigma_j;\g_\rho), \mathbf{c}_{\Sigma_j}, \mathbf{h}_{\Sigma_j}) \cdot  \mathrm{tor}(\mathcal{G}). \nonumber
			 \end{align} 
			where $\mathrm{tor}(\mathcal{G})$ is the Reidemeister torsion of the long exact sequence \eqref{eqn:les2} with respect to $\mathbf{h}_{M_j}$, $\mathbf{h}_Y$, $\mathbf{h}_{\Sigma_j}$, and  $\mathbf{h}_M$. 
			It is clear  from the definition \eqref{eqn:v} that $v$ is even since the number of $i$-dimensional cells in each $M$, $M_j$, $Y$ and $\Sigma_j$ is even for all $i$.
			Also, a direct computation from the definition \eqref{eqn:u} gives that  $u \equiv n$ in modulo 2.
			Recall that there are two trivial terms $H^0(M;\g_\rho)$ and $H^3(M;\g_\rho)$ in $\mathcal{G}$. Ignoring these trivial terms, we rewrite $\mathcal{G}$ as
%			\begin{align}
%				0 & \longrightarrow   \mathcal{G}^0  \overset{g_0}{\longrightarrow}  \mathcal{G}^1
%				\overset{d_0}{\longrightarrow} \mathcal{G}^2 \overset{f_1}{\longrightarrow} \mathcal{G}^3 \overset{g_1}{\longrightarrow} \mathcal{G}^4  \overset{d_1}{\longrightarrow} \mathcal{G}^5
%				\overset{f_2}{\longrightarrow} \mathcal{G}^6
%				\overset{g_2}{\longrightarrow} \mathcal{G}^7
%				\longrightarrow 0. 
%			\end{align}
			\begin{center}
				\begin{tikzcd}[column sep=0.6cm]
					0 \arrow[r]  & \mathcal{G}^0 \arrow[r,"G_0"] \arrow[d, "\simeq"] & \mathcal{G}^1 \arrow[r,"D_0"]  \arrow[d, "\simeq"]& \mathcal{G}^2 \arrow[r,"F_1"] \arrow[d, "\simeq"] &  \mathcal{G}^3 \arrow[r,"G_1"]  \arrow[d, "\simeq"] & \mathcal{G}^4 \arrow[r,"D_1"]  \arrow[d, "\simeq"] & \mathcal{G}^5  \arrow[r,"F_2"] \arrow[d, "\simeq"] & \mathcal{G}^6 \arrow [r,"G_2"] \arrow[d, "\simeq"] & \mathcal{G}^7 \arrow[r] \arrow[d, "\simeq"] & 0\\ 
					& \Cbb^{k+1} \arrow[r,"G_0"]& \Cbb^n \arrow[r,"D_0"] & \Cbb^{n-k} \arrow[r,"F_1"] & \Cbb^{2n+1} \arrow[r,"G_1"] & \Cbb^{2n} \arrow[r,"D_1"]&  \Cbb^{n-k} \arrow[r,"F_2"]& \Cbb^{2n-k} \arrow[r,"G_2"] & \Cbb^n &
				\end{tikzcd}
			\end{center}
		where the first and second rows are identified with respect to $\mathbf{h}_{M_j}$, $\mathbf{h}_Y$, $\mathbf{h}_{\Sigma_j}$, and $\mathbf{h}_M$.
		We	choose a tuple $b^i$ of vectors in $\mathcal{G}^i$ as
			\begin{align*}
				& b^0 = (\mathfrak{e}_1,\mathfrak{e}_2,\ldots, \mathfrak{e}_{k+1}),\ b^1=(\mathfrak{e}_{k+1}, \mathfrak{e}_{k+2},\ldots, \mathfrak{e}_{n-1}), \ b^2 = \mathfrak{e}_{1}\\
				& b^3=(\mathfrak{e}_1,\mathfrak{e}_2,\ldots, \mathfrak{e}_n,\, \mathfrak{e}_{n+2},\mathfrak{e}_{n+3},\ldots,\mathfrak{e}_{2n+1}),\ b^4 = \emptyset, \\
				& b^5=(\mathfrak{e}_1, \mathfrak{e}_2,\ldots, \mathfrak{e}_{n-k}),\  b^6=(\mathfrak{e}_{n-k+1},\mathfrak{e}_{n-k+2},\ldots,\mathfrak{e}_{2n-k}),\  b^7=\emptyset
			\end{align*}
			where $\mathfrak{e}_k$ is a unit vector whose coordinates are all zero, except one at the $k$-th coordinate.
			Then the basis transition matrix $A_i$ at $\mathcal{G}^i$ (see Section \ref{sec:cw}) is given by
			\begin{align*}
				& A_0=I_{k+1}, \ 
				A_1=\begin{pmatrix}
					I_k & \rvline & \begin{matrix} -1 \\[-5pt] \vdots \\[-3pt] -1 \end{matrix} & \rvline & 0 \\
					\hline
					0 & \rvline & \begin{matrix} -1 \\[-5pt] \vdots \\[-3pt] -1 \end{matrix} & \rvline & I_{n-k-1} \\
					\hline
					0 & \rvline & -1& \rvline & 0 
				\end{pmatrix},\ 
			A_2=\begin{pmatrix} 
				0 & 1 \\
				I_{n-k-1} & 0
			\end{pmatrix}, \
		A_3=\begin{pmatrix}
			\begin{matrix} 1 \\[-5pt] \vdots \\[-3pt] 1 \end{matrix} & \rvline & I_n & \rvline & 0 \\
			\hline
			1 & \rvline & 0 & \rvline & 0\\
			\hline
			\begin{matrix} \kappa_1 \\[-5pt] \vdots \\[-3pt] \kappa_n \end{matrix} & \rvline & 0& \rvline & I_n
		\end{pmatrix},\\
			& A_4=\begin{pmatrix}
			1 & \cdots & 0 & \rvline  & 0 & \cdots & 0 \\[-2pt]
			\kappa_1&   & 0 & \rvline   & -1 &  & 0\\[-2pt]
			\vdots & \ddots& \vdots & \rvline & \vdots & \ddots & \vdots \\[-2pt]
			0 &  & 1 &  \rvline &  0 & & 0  \\[-2pt]
			0 & \cdots & \kappa_n  & \rvline& 0 & \cdots & -1 
		\end{pmatrix},\ A_5= I_{n-k}, \ 
		A_6=\begin{pmatrix}
		I_{n-k} & \rvline & 0 \\
		\hline
		\begin{matrix}
			0 \\
			I_{n-k}
		\end{matrix} & \rvline &  I_{n} 
	\end{pmatrix},\ A_7=-I_n.
			\end{align*}
		It follows that  $\mathrm{tor}(\mathcal{G}) = (-1)^{n-k} \, (-1)^{n-k-1}\,   (-1)^{n}\, (-1)^{\frac{n(n+1)}{2}} (-1)^n=(-1)^{\frac{n(n+1)}{2}+1 }$. 
		Therefore, we conclude that 
			\begin{align}
			&e \cdot \prod_{j=1}^n\mathrm{Tor}(C^\ast(M_j;\g_\rho), \mathbf{c}_{M_j}, \mathbf{h}_{M_j}) \cdot \mathrm{Tor}(C^\ast(Y;\g_\rho), \mathbf{c}_{Y}, \mathbf{h}_{Y}) \label{eqn:proof2} \\
		&=(-1)^{\frac{n(n+1)}{2}+n+1} \, \mathrm{Tor}(C^\ast(M;\g_\rho), \mathbf{c}_{M}, \mathbf{h}_{M}) 	\prod_{j=1}^n\mathrm{Tor}(C^\ast(\Sigma_j;\g_\rho), \mathbf{c}_{\Sigma_j}, \mathbf{h}_{\Sigma_j}). \nonumber
	\end{align} 		
		On the other hand, applying the formula \eqref{eqn:gluing} to the short exact sequence 
		\begin{equation} \label{eqn:ses3}  0 \rightarrow \bigoplus_{j=1}^n C_\ast(\Sigma_j;\Rbb) \rightarrow \bigoplus_{j=1}^n C_\ast(M_j;\Rbb) \oplus C_\ast(Y;\Rbb) \rightarrow C_\ast(M;\Rbb) \rightarrow 0,
		\end{equation}
		we have 
		\begin{align}
			&e \cdot \prod_{j=1}^n\mathrm{Tor}(C_\ast(M_j;\Rbb), c_{M_j}, h_{M_j}) \cdot \mathrm{Tor}(C_\ast(Y;\Rbb), c_{Y}, h_{Y}) \label{eqn:les3} \\
			&=(-1)^{u'+v'} \prod_{j=1}^n\mathrm{Tor}(C_\ast(\Sigma_j;\Rbb), c_{\Sigma_j}, h_{\Sigma_j}) \cdot \mathrm{Tor}(C_\ast(M;\Rbb), c_{M}, h_{M}) \cdot \mathrm{tor}(\mathcal{G}'). \nonumber
		\end{align} 
		where $\mathrm{tor}(\mathcal{G}')$ is the Reidemeister torsion of the long exact sequence induced from \eqref{eqn:ses3} 
%		\begin{align*}
%			0 &\rightarrow \bigoplus_{j=1}^n H_0(\Sigma_j;\Rbb) \rightarrow \bigoplus_{j=1}^n H_0(M_j;\Rbb) \oplus H_0(Y;\Rbb) \rightarrow H_0(M;\Rbb) \\
%			& \rightarrow \bigoplus_{j=1}^n H_1(\Sigma_j;\Rbb) \rightarrow \bigoplus_{j=1}^n H_1(M_j;\Rbb) \oplus H_1(Y;\Rbb) \rightarrow H_1(M;\Rbb) \\
%			&\rightarrow \bigoplus_{j=1}^n H_2(\Sigma_j;\Rbb) \rightarrow H_2(Y;\Rbb)
%		\end{align*}
		with respect to the bases $h_{\Sigma_j}$, $h_{M_j}$, $h_Y$, and $h_M$.
		Repeating similar computations, we have $u'\equiv v' \equiv 0$ in modulo 2 and $\mathrm{tor}(\mathcal{G}') = (-1)^{\frac{n(n+1)}{2}}$.
		Then from the equation \eqref{eqn:les3} we have
		\begin{equation} \label{eqn:sign}
			e \cdot \prod_{j=1}^n\epsilon(o_{M_j}) \cdot \epsilon(o_{Y}) = (-1)^\frac{n(n+1)}{2} \prod_{j=1}^n\epsilon(o_{\Sigma_j}) \cdot \epsilon(o_{M}).
		\end{equation}
		Combining the equations  \eqref{eqn:proof2} and \eqref{eqn:sign} with Example \ref{exam:torus} and Proposition \ref{lem:YTor}, we obtain the desired formula.
		\end{proof}

		\subsection{Proofs of Theorems \ref{thm:main0} and \ref{thm:main}}	\label{sec:3.3}
		Recall that $\mathcal{X}(M)$ is the character variety of irreducible representations  $\pi_1(M) \rightarrow \G$ and $\mathcal{X}_\mu^c(M)$ is the pre-image of $c \in \Cbb$ under the trace function $\mathcal{X}(M) \rightarrow \Cbb$ of $\mu$. We use the notations $\mathcal{X}(M_j)$ and $\mathcal{X}_\mu^c(M_j)$ similarly for $1 \leq j \leq n$.
		Since we assumed that
			\begin{itemize} 
%			\item[(a)] the character variety $\mathcal{X}(M_j)$ consists of 1-dimensional components;
		 \item[(C)] the level set $\mathcal{X}_\mu^c(M_j)$ consists of finitely many $\mu$-regular characters with the canonical longitude having trace other than $\pm2$ for generic $c \in \Cbb$,
		\end{itemize}
		the conditions \eqref{eqn:assum} and \eqref{eqn:assum2} in Section \ref{sec:gluing}  are satisfied for generic $c \in \Cbb$.
		It follows that the adjoint Reidemeister torsion is well-defined on the level set $\mathcal{X}_\mu^c(M)$ for generic $c\in\Cbb$.
	
%		For $M= \Mcal$ or $M_j$ we denote by $\mathcal{X}(M)$ the character variety of irreducible representations $\pi_1(M)\rightarrow \G$ and let $\mathcal{X}_\mu^c(M)$ be the pre-image of $c \in \Cbb$ under the trace function $\tr_\mu:\mathcal{X}(M) \rightarrow \Cbb$ of $\mu$.
%		We assume that		
%		\begin{itemize} 
%			\item  $\mathcal{X}(M_j)$ consists of 1-dimensional components
%			\item $\mathcal{X}_\mu^c(M_j)$ consists of (finitely many) $\mu$-regular characters with $l_j \neq \pm1 $ for generic $c \in \Cbb$
%		\end{itemize}
%		for all $1 \leq j \leq n$. We say that $M_j$ satisfies the vanishing identity if
%			\begin{equation} \label{eqn:vanishing}
%				\sum_{\mathcal{X}_\mu^c(M_j)} \frac{1}{\tau_\mu(M_j;\rho)} =0
%			\end{equation}
%		for generic $c \in \Cbb$.
			\begin{lemma} 
			\label{lemma:fibers}
			The connected components of  $\mathcal{X}_\mu^c(M)$
			are the pre-images of the restriction map (i.e.~induced by the inclusions $M_j \rightarrow M$):
			\[\Phi : \mathcal{X}_\mu^c(M) \rightarrow \prod_{j=1}^n \left( \mathcal{X}_{\mu}^c(M_j) \sqcup \{[\alpha_j]\} \right)
			\setminus\{( [\alpha_1] ,\ldots,[\alpha_n])\}
			\]
			where $\alpha_j : \pi_1(M_j)\rightarrow \G$ is the abelian representation with $\tr (\alpha_j(\mu)) =c$
			\end{lemma}	
%					Those fibers are bending orbits of dimension $(n-k-1)$, where $0\leq k\leq n-1$ is the number $M_j$ whose restriction is abelian.
		
		\begin{proof} We first prove that  $\Phi$ is surjective.  Let $\rho_j$ be a representation $\pi_1(M_{j}) \rightarrow \G$ satisfying $\tr(\rho_j(\mu))=c$ for $1 \leq j \leq n$. Since we assume that $c \neq \pm2$,  we can conjugate each
			$\rho_j$ so that
			$\rho_1(\mu)= \rho_2(\mu)=\cdots =\rho_n(\mu)$.
			This is sufficient to extend these representations to 
			$\rho:\pi_1(M)\rightarrow \G$ which is irreducible since at least one of $\rho_j$'s is irreducible.
			
			A point, say $p$, in the image of $\Phi$ is $([\alpha_1],\ldots,[\alpha_k], [\rho_{k+1}],\dotsc,[\rho_n])$ up to reordering where $\alpha_1\dots,\alpha_k$ are abelian and $\rho_{k+1},\dotsc,\rho_n$  are irreducible.
			To analyze the pre-image  $\Phi^{-1}(p)$, consider 
			two characters in $\Phi^{-1}(p)$,  those are 
			conjugacy classes of irreducible representations $\rho$ and $\rho'$ of $\pi_1(M)$. As 
			$\tr(\rho(\mu))=\tr(\rho'(\mu))\neq \pm 2$, after conjugating
			we may assume that 
			$\rho(\mu)=\rho'(\mu)$.
			Let $D\subset \mathrm{PSL}_2(\mathbb C) $ denote the centralizer of  $\rho(\mu)=\rho'(\mu)$. Since 
			$\tr(\rho(\mu))=\tr(\rho'(\mu))\neq \pm 2$, the group  $D$ is conjugate to the group of diagonal matrices and thus 
			$D\cong \mathbb{C}^*$.
			Let $\rho_j$ and $\rho_j'$ denote the respective restrictions of $\rho$ and $\rho'$ to $\pi_1(M_j)$.
			% % 		   After conjugating
			% % 		   we may assume that 
			% % 		   $\rho(\mu)=\rho'(\mu)$.
			The assumption $\rho(\mu)=\rho'(\mu)$ implies that:
			\begin{itemize}
				\item For $j=1,\ldots,k$, $\rho_j=\rho_j'$. It is because of that
				the genericity assumption \eqref{eqn:assum} implies that
				$\rho_j$ and $\rho_j'$ are abelian, and an abelian representation of a knot exterior is determined by the trace of $\mu$. 
				\item For $j=k+1,\ldots,n$, $\rho_j'$ and $\rho_j$ are conjugate by some  matrix of $D$, because an irreducible representation is determined by its character.
			\end{itemize}
			Namely, $\rho$ and $\rho'$ differ by \emph{bending}  along some of the tori $\Sigma_{k+1},\dotsc,
			\Sigma_{n}$. Note that bending along  all tori $\Sigma_{k+1},\dotsc,\Sigma_{n}$ simultaneously by the same matrix in $D$ 
			does not change the conjugacy class. 
			It follows that the pre-image $\Phi^{-1}(p)$ is homeomorphic to 
			\[\underbrace{(D\times \cdots \times D)}_{n-k}/D\cong 
			\underbrace{(\Cbb^\ast \times \cdots \times \Cbb^\ast)}_{n-k}/\Cbb^\ast  
			\cong (\mathbb{C}^*)^{n-k-1}.\]	As the pre-images of $\Phi$ are connected and the image is discrete, those
			pre-images are the connected components.
		\end{proof}
		
		From Theorem~\ref{lem:torproduct} and Lemma~\ref{lemma:fibers}, we obtain Theorem \ref{thm:main0}: the adjoint Reidemeister torsion is locally constant on $\mathcal{X}_\mu^c(M)$.
		Note that the term $(m-m^{-1})^{2n-2}$ in Theorem~\ref{lem:torproduct} is the constant $(c^2-4)^{n-1}$ on $\mathcal{X}_\mu^c(M)$.
		
		On the other hand, we have
		\begin{align*}
			\frac{1}{(c^2-4)^{n-1}}\sum_{[\rho] \in \mathcal{X}_\mu^c(M)} \frac{1}{\tau_\mu(M;\rho)} 
			&=
		\prod_{j=1}^n\left(\sum_{[\rho] \in \mathcal{X}_\mu^c(M_j)} \frac{1}{\tau_\mu(M_j;\rho)} + \frac{1}{\tau_\mu(M_j;\alpha_j)} \right)-\prod_{j=1}^n \frac{1}{\tau_\mu(M_j;\alpha_j)}\\
		&=\sum_{J \subsetneq \{1,\ldots,n\}}\left( \prod_{j \notin J}  \sum_{[\rho]\in\mathcal{X}_\mu^c (M_j)}\frac{1}{\tau_\mu(M_j;\rho)}\right) \cdot \prod_{j\in J} \frac{1}{\tau_\mu(M_j;\alpha_j)}. 
		\end{align*}
	Here $J$ runs on all subsets of $\{1,\ldots,n\}$ different from the whole set ($J$ is the subset of indexes $j$   
	such that the restriction to $\pi_1(M_j)$ is abelian, hence $J$ may be empty but not the whole set). The notation $[\rho] \in \mathcal{X}_\mu^c(M)$ means that we take one representative on each connected component of $X_\mu^c(M)$, which  agrees with the ordinary sum for $M_j$.
	This completes the proof of Theorem \ref{thm:main}, because we have assumed that for each $j=1,\ldots, n$:
	\[
	\sum_{[\rho]\in\mathcal{X}_\mu^c (M_j)}\frac{1}{\tau_\mu(M_j;\rho)} =0.
	\]
%	directly from the equation \eqref{eqn:imp}.

	\section*{Acknowledgments} 
	We would like to thank Stavros Garoufalidis for his helpful comments on a draft version of the paper. The first author has been partially supported by the Spanish
    Micinn/FEDER  grant
    PGC2018-095998-B-I00.
    The second author was supported by Basic Science Research Program through the NRF of Korea funded by the Ministry of Education (2020R1A6A3A03037901).

		\bibliographystyle{alpha}
		\bibliography{main}

	\end{document}

%% file: torus.pdf_tex
%% Creator: Inkscape 1.0.2 (e86c8708, 2021-01-15), www.inkscape.org
%% Accompanies image file 'torus.pdf' (pdf, eps, ps)
\begingroup%
  \makeatletter%
  \providecommand\color[2][]{%
    \errmessage{(Inkscape) Color is used for the text in Inkscape, but the package 'color.sty' is not loaded}%
    \renewcommand\color[2][]{}%
  }%
  \providecommand\transparent[1]{%
    \errmessage{(Inkscape) Transparency is used (non-zero) for the text in Inkscape, but the package 'transparent.sty' is not loaded}%
    \renewcommand\transparent[1]{}%
  }%
  \providecommand\rotatebox[2]{#2}%
  \newcommand*\fsize{\dimexpr\f@size pt\relax}%
  \newcommand*\lineheight[1]{\fontsize{\fsize}{#1\fsize}\selectfont}%
  \ifx\svgwidth\undefined%
    \setlength{\unitlength}{254.68285268bp}%
    \ifx\svgscale\undefined%
      \relax%
    \else%
      \setlength{\unitlength}{\unitlength * \real{\svgscale}}%
    \fi%
  \else%
    \setlength{\unitlength}{\svgwidth}%
  \fi%
  \global\let\svgwidth\undefined%
  \global\let\svgscale\undefined%
  \makeatother%
  \begin{picture}(1,0.33390301)%
    \lineheight{1}%
    \setlength\tabcolsep{0pt}%
    \put(0,0){\includegraphics[width=\unitlength,page=1]{torus.pdf}}%
    \put(0.05201632,0.01699869){\color[rgb]{0,0,0}\makebox(0,0)[lt]{\lineheight{1.25}\smash{\begin{tabular}[t]{l}$p$\end{tabular}}}}%
    \put(0.19836288,0.15809528){\color[rgb]{0,0,0}\makebox(0,0)[lt]{\lineheight{1.25}\smash{\begin{tabular}[t]{l}$\Sigma$\end{tabular}}}}%
    \put(0.20243703,0.0087265){\color[rgb]{0,0,0}\makebox(0,0)[lt]{\lineheight{1.25}\smash{\begin{tabular}[t]{l}$\lambda$\end{tabular}}}}%
    \put(0.02615864,0.14936706){\color[rgb]{0,0,0}\makebox(0,0)[lt]{\lineheight{1.25}\smash{\begin{tabular}[t]{l}$\mu$\end{tabular}}}}%
    \put(0,0){\includegraphics[width=\unitlength,page=2]{torus.pdf}}%
    \put(0.57469724,0.01099575){\color[rgb]{0,0,0}\makebox(0,0)[lt]{\lineheight{1.25}\smash{\begin{tabular}[t]{l}$\widetilde{p}$\end{tabular}}}}%
    \put(0.74154586,0.15517186){\color[rgb]{0,0,0}\makebox(0,0)[lt]{\lineheight{1.25}\smash{\begin{tabular}[t]{l}$\widetilde{\Sigma}$\end{tabular}}}}%
    \put(0.56242235,0.14817122){\color[rgb]{0,0,0}\makebox(0,0)[lt]{\lineheight{1.25}\smash{\begin{tabular}[t]{l}$\widetilde{\mu}$\end{tabular}}}}%
    \put(0.74900414,-0.00211207){\color[rgb]{0,0,0}\makebox(0,0)[lt]{\lineheight{1.25}\smash{\begin{tabular}[t]{l}$\widetilde{\lambda}$\end{tabular}}}}%
    \put(0,0){\includegraphics[width=\unitlength,page=3]{torus.pdf}}%
  \end{picture}%
\endgroup%